\documentclass[12pt]{article}%
\usepackage{amssymb}
\usepackage{amsfonts}
\usepackage{graphicx, epsfig}
\usepackage{amsmath}
\usepackage{graphicx}
\usepackage{color}
\usepackage{rotating}
\usepackage[authoryear]{natbib}
\usepackage[left=1.3cm, right=1.3cm]{geometry}
\usepackage{hyperref}%
\usepackage{authblk}
\usepackage{placeins}

\newtheorem{theorem}{Theorem}

\newtheorem{condition}{Condition}

\newtheorem{lemma}{Lemma}

\newcommand{\zb}{\mathbf{z}}

\newcommand{\fb}{\mathbf{f}}
\newcommand{\eb}{\mathbf{e}}

\newcommand{\Bb}{\boldsymbol{\beta}}
\newcommand{\chib}{\boldsymbol{\chi}}
\newcommand{\Zb}{\mathbf{Z}}
\newcommand{\Fb}{\mathbf{F}}
\newcommand{\Eb}{\mathbf{E}}

\newcommand{\halb}{\widehat{\boldsymbol{\alpha}}}
\newcommand{\hBb}{\widehat{\boldsymbol{\beta}}}
\newcommand{\hfb}{\widehat{\mathbf{f}}}
\newcommand{\cfb}{\check{\mathbf{f}}}
\newcommand{\calb}{\check{\boldsymbol{\alpha}}}
\newcommand{\cBb}{\check{\boldsymbol{\beta}}}

\newcommand{\hk}{\widehat{k}}
\newcommand{\hFb}{\widehat{\mathbf{F}}}

\newcommand{\hzb}{\widehat{\mathbf{z}}}
\newcommand{\hZb}{\widehat{\mathbf{Z}}}

\DeclareMathOperator{\Tr}{Tr}
\newenvironment{proof}[1][Proof]{\textbf{#1.} }{\ \rule{0.5em}{0.5em}}
\begin{document}

\title{Consistency of Generalized Dynamic Principal Components in Dynamic Factor Models}
\author{Ezequiel Smucler}
\affil{Department of Statistics, University of British Columbia}
\date{}
\maketitle

\begin{abstract}
We study the theoretical properties of the generalized dynamic principal components introduced in \cite{PenaYohai2016}. In particular, we prove that when the data follows a dynamic factor model, the reconstruction provided by the procedure converges in mean square to the common part of the model as the number of series and periods diverge to infinity. The results of a simulation study support our findings. 
\end{abstract}


\section{Introduction}

\label{sec:intro}

The increased availability of data in all fields of modern science has highlighted the importance of statistical methods for dimension reduction. Dimension reduction is specially important when dealing with vector time series, since the usual parametric models have a number of parameters that increases with the square of the number of series.
The two classical dimension reduction techniques, factor models and principal components have their dynamic counter parts in dynamic factor models (DFM) and dynamic principal components (DPC).

Dynamic factor models aim at accounting for cross-correlations in the data, both the instantaneous and the lagged ones. 
These models are especially popular among econometricians and have generated large amounts of applied work, see \cite{Forni2015} for some references on applications.
Important recent references on DFM include \cite{Forni2000}, \cite{Forni2005}, \cite{Forni2015},  \cite{SW2002} and \cite{BaiNg}.

On the other hand, the goal of DPC is to find an unobserved lower dimensional vector time series that, in some specific sense, is able to optimally reconstruct the observed vector time series. The first proposal for DPC is that of \cite{Brillinger}, based on spectral techniques. \cite{PenaYohai2016} proposed a new, entirely data-analytic approach called generalized dynamic principal components (GDPC), in which the dynamic principal component need not be a linear combination of the data. Even though they did not provide any theory for GDPC, in extensive simulation studies they showed that the proposed procedure has excellent reconstruction properties. Recently, \cite{PenaSmuclerYohai} proposed one-sided dynamic principal components (ODPC) for vector time series, defined as linear combinations of the present and past values of the series that minimize the reconstruction mean squared error. Unlike Brillinger's DPC and the GDPC, the ODPC are useful for forecasting.

Both dynamic dimension reduction techniques, DFM and DPC, are closely related, in that both assume that an important part of the series can be explained by a lower dimensional structure. Moreover, \cite{Forni2000} and \cite{Forni2005} showed that the DPC proposed in \cite{Brillinger} can be used to consistently estimate the common part of a DFM, when both $T$ and $m$, the number of periods and the number of series, diverge to infinity at any rate. \cite{PenaSmuclerYohai} showed that when the data follows a DFM, the reconstruction provided by the ODPC converges in mean square to the common part of a  DFM. Their result requires a two step asymptotic, where first $m$ is kept fixed and $T\to \infty$ and then $m \to \infty$.
On the other hand, \cite{SW2002} proved that the ordinary principal components of the data can be used to consistently estimate the factors in a static factor model, when both $T$ and $m$ diverge at any rate.

In this short note, we further explore the relation between DFM and DPC by studying the theoretical properties of GDPC when the data follows a DFM. In Section \ref{sec:gdpc} we recall the definition of GDPC. In Section \ref{sec:cons:dfm} we state our assumptions and our main result: if the data follows a dynamic factor model, the reconstruction provided by GDPC converges in mean square to the common part of the factor model when both $T$ and $m$ diverge at any rate. Moreover, the convergence is with rate at least $\min\left(T^{1/4}, m^{1/4} \right)$. 
In Section \ref{sec:sim} we report the results of a simulation study that compares the estimation of the common part obtained using GDPC with the estimations obtained using procedures based on DFMs. Some conclusions are provided in Section \ref{sec:conclu}. Proofs can be found in Section \ref{sec:proofs}, a technical appendix.

\section{Generalized dynamic principal components}
\label{sec:gdpc}
Consider a vector time series $\zb_{1}, \dots, \zb_{T}$, where $\zb_{t}^{\prime}=(z_{t,1},\dots,z_{t,m})$. 
The aim of generalized dynamic principal components as defined by \cite{PenaYohai2016} is to reconstruct $\zb_{t}$, $t=1,\dots,T$, using a linear combination of a lower-dimensional unobserved time series and its lags.
Let $\cfb^{\prime}=(\check{f}_{1-k}, \dots, \check{f}_{T})$, $\calb^{\prime}=(\check{\alpha}_{1}, \dots, \check{\alpha}_{m})$ and $\cBb \in \mathbb{R}^{(k+1)\times m}$ be the matrix with entries $\check{\beta}_{h,j}$.
We can define a reconstruction of $\zb_{t}$ using $\cfb, \calb$ and $\cBb$ as
\begin{align*}
\widehat{z}_{t,j}(\cfb, \calb,\cBb) = \check{\alpha}_{j} + \sum \limits_{h=0}^{k}\check{\beta}_{h,j}\check{f}_{t-h}, \quad j=1,\dots,m.
\end{align*}
Let $\hzb_{t}(\cfb, \calb,\cBb)$, $t=1,\dots,T$, be the resulting reconstruction and let 
\begin{align*}
\text{MSE}(\cfb, \calb, \cBb)= \frac{1}{T m} \sum\limits_{t=1}^{T}\Vert \zb_{t} - \hzb_{t}(\cfb, \calb,\cBb)\Vert^{2}
\end{align*}
be the reconstruction mean squared error.
\cite{PenaYohai2016} define the first generalized dynamic principal component  of the data using $k\geq 0$ lags as the vector $\hfb$ with zero mean and unit variance such that for some optimal $\halb \in \mathbb{R}^{m}$, $\hBb \in \mathbb{R}^{(k+1)\times m}$ the reconstruction mean squared error is minimal.
They propose an iterative algorithm for the computation of GDPC which solves two sets of linear equations at each iteration. The algorithm is specially suited for high-dimensional problems, since the dimensions of the matrices to be inverted do not depend on $m$.

\section{Consistency in the dynamic factor model}
\label{sec:cons:dfm}

Suppose we have observations, $\zb_{1}, \dots, \zb_{T}$, $\zb_{t}^{\prime}=(z_{t,1},\dots,z_{t,m})$, of a double indexed stochastic process $\lbrace z_{t,j} : t\in\mathbb{Z}, j \in \mathbb{N} \rbrace$.
Consider a dynamic factor model with one dynamic factor, $f_{t}$, and finite dimensional factor space, more precisely
\begin{align}
z_{t,j}=\sum\limits_{h=0}^{k}\beta_{h,j}f_{t-h}+e_{t,j},\quad t=1,\dots
,T,\quad j=1,\dots,m,
\nonumber
\end{align}
where the $z_{t,j}$ are the observed variables, $e_{t,j}$ for $j=1,...,m$ and $f_{t}$ are random variables and
$\beta_{h,j}$ are factor loadings. This can be expressed in the form of a static factor
model with $r=k+1$ static factors, as
\begin{align}
\mathbf{z}_{t}=\Bb^{\prime}\mathbf{f}_{t}+\mathbf{e}_{t},\quad
t=1,\dots,T,
\nonumber
\end{align}
where $\mathbf{e}_{t}=(e_{t,1},\dots,e_{t,m})^{\prime}$, $\Bb
\in\mathbb{R}^{(k+1)\times m}$ is the matrix with entries $\beta_{h,j}$ and
$\mathbf{f}_{t}=(f_{t},\dots,f_{t-k})^{\prime}$. 
The term $\chib_{t} = \Bb^{\prime}\mathbf{f}_{t} $ is called the common part of the model. $\eb_{t}$ is called the idiosyncratic part, which is assumed to have weak cross-correlations. Moreover, we will assume
\begin{condition}
$ $

\begin{itemize}
\label{condition:DFM:exist}

\item[(a)] For each $m$, $\chib_{t}$ and $\eb_{t}$ are zero-mean, second order $m$-dimensional stationary process that have a spectral density.
\item[(b)] $\fb_{t}$ is second-order stationary and $\fb_{t}$ and $\eb_{t}$ are uncorrelated for all $t$. 
\item[(c)] Let $\boldsymbol{\Sigma}^{\chib}, \boldsymbol{\Sigma}^{\eb}$ be the covariance matrices of $\chib_{t}$ and $\eb_{t}$. 
Let $\lambda_{m,j}^{\chib}, \lambda_{m,j}^{\eb}$ be their $j$-th eigenvalues, $\lambda_{j}^{\chib} = \sup_{m \in\mathbb{N}} \lambda_{m,j}^{\chib}$ and $\lambda_{j}^{\eb} = \sup_{m \in\mathbb{N}} \lambda_{m,j}^{\eb}$. Then  $\lambda_{k+1}^{\chib}=\infty$ and  $\lambda_{1}^{\eb}<\infty$.
\end{itemize}
\end{condition}

Theorem B of \cite{Forni2015}, citing an earlier result of \cite{Chamberlain}, implies that under Condition \ref{condition:DFM:exist}, the decomposition $\zb_{t}= \chib_{t} + \eb_{t}$ is unique. The following additional assumptions will be needed.

\newpage

\begin{condition}
$ $

\begin{itemize}
\label{condition:DFM:consis}

\item[(a)] $(\Bb\Bb^{\prime})/m\rightarrow\mathbf{I}_{k+1}$.
\item[(b)] $\limsup_{m}\sum_{u\in\mathbb{Z}}\left(  \mathbb{E} \eb_{t}^{\prime} \eb_{t+u} / m \right)^{2}< \infty.$
\item[(c)] $\limsup_{m}  (1/m) \sup_{t, s\in\mathbb{Z}} \sum_{j=1}^{m}\sum_{i=1}^{m} \left\vert cov (e_{t,i}e_{s,i}, e_{t,j}e_{s,j})\right\vert < \infty$.
\end{itemize}

\end{condition}

This is similar to the set of assumptions used by \cite{SW2002}. Condition \ref{condition:DFM:consis} (a)
is a standarization assumption and (b)-(c) allow for weak cross-sectional and temporal correlations in
the idiosyncratic part. 

Theorem \ref{theo:consis} shows that, asymptotically as $T$ and $m$ go to infinity, the reconstruction provided by the GDPCs converges in mean square to the common part of the model.

\begin{theorem}
\label{theo:consis}
Assume Conditions \ref{condition:DFM:exist} and \ref{condition:DFM:consis} hold. Let $(\hfb, \halb, \hBb)$ be the first GDPC with corresponding intercepts and loadings, defined using $\hk = k$ lags. 
Then, as $T \to \infty$ and $m \to \infty$,
\begin{align*}
\frac{1}{Tm}\sum\limits_{t=1}^{T}\Vert \chib_{t} - \hzb_{t}(\hfb, \halb,\hBb)\Vert^{2} = O_{P}\left(\frac{1}{\min\left(T^{1/4}, m^{1/4} \right)} \right)
\end{align*}
\end{theorem}

\section{Simulations}
\label{sec:sim}

We compare the estimations of the common part of dynamic factor models obtained using GDPC, FHLR \citep{Forni2005}, and SW \citep{SW2002}.
We generated $300$ replications of the dynamic factor model
$$
z_{t,j} = \chi_{t,j} + e_{t,j} = c(\beta_{0,j} f_{t} + \dots+\beta_{k,j} f_{t-k}) + e_{t,j}, \quad 1\leq t\leq T, 1\leq j \leq m,
$$
for different choices of the loadings, the factors, the idiosyncratic shocks and all combinations of $T\in\lbrace 100, 200\rbrace$, $m \in \lbrace 100, 200, 400\rbrace$. In all cases, $c$ is chosen so that the mean sample variance of the common part is equal to one. The matrix $\Bb$ with entries $\beta_{h,j}$ is generated at random as follows: we generate a matrix with i.i.d. standard normal random variables and orthogonalize it so that $\Bb \Bb^{\prime} = m I_{k+1}$. We consider the following choices of the number of lags $k$, the distribution of the dynamic factor and of the idiosyncratic part:
\begin{itemize}
\itemindent=0.35in
\item[DFM1] $(f_{t})_{t}$ satisfies a MA(1) model, with i.i.d. standard normal innovations and parameter generated at random uniformly on $(-0.9, 0.9)$. We take $k=1$. The idiosyncratic part is formed by i.i.d. normal random variables.
\item[DFM1AR] We incorporate dynamics into the idiosyncratic part of DFM1, by letting the idiosycratic part be formed by independent AR(1) processes with parameters generated at random uniformly on $(-0.9, 0.9)$ and standardized to unit population variance.
\item[DFM2] $(f_{t})_{t}$ satisfies an AR(1) model, with i.i.d. standard normal innovations and parameter generated at random uniformly on $(-0.9, 0.9)$.  We take $k=2$. The idiosyncratic part is formed by i.i.d. normal random variables.
\item[DFM2AR] As before, we incorporate dynamics into the idiosyncratic part of DFM2, by letting the idiosycratic part be formed by AR(1) processes.
\end{itemize}

In all cases, we assume that the number of dynamic factors and lags is known.
We computed GDPC using the \texttt{gdpc} \texttt{R} package \citep{gdpc-package}. We used our own implementation of the SW procedure and a MATLAB implementation of FHLR kindly provided by the authors. The lag window size for the FHLR procedure was taken as $[\sqrt{T}]$.
The SW estimate of the common part is obtained by projecting the data on the space spanned the first $r$ ordinary principal components.
Let $\chib$ be the common part of the dynamic factor model in one of the replications and let $\widehat{\chib}$ be an estimation of it. Performance is measured by $
\Vert \chib - \widehat{\chib} \Vert_{F}^{2}/\Vert \chib \Vert_{F}^{2}$,
which is then averaged over all replications. Results are shown in Table \ref{tab:sim:DFM:MSE}
\begin{table}[ht]
\centering
\begin{tabular}{llrrrrrr}
\hline
  &   &  DFM1 &     &    & DFM1AR &    & \\
  \hline
$T$ & $m$ & GDPC & FHLR & SW & GDPC & FHLR & SW \\ 
  \hline
100 & 100 & 0.0406 & 0.0495 & 0.0508 & 0.0537 & 0.0615 & 0.0621 \\ 
   & 200 & 0.0353 & 0.0399 & 0.0403 & 0.0476 & 0.0512 & 0.0512 \\ 
   & 400 & 0.0338 & 0.0352 & 0.0351 & 0.0437 & 0.0468 & 0.0462 \\ 
  200 & 100 & 0.0253 & 0.0346 & 0.0353 & 0.0311 & 0.0408 & 0.0411 \\ 
   & 200 & 0.0217 & 0.0251 & 0.0251 & 0.0261 & 0.0314 & 0.0311 \\ 
   & 400 & 0.0176 & 0.0204 & 0.0201 & 0.0232 & 0.0262 & 0.0256 \\ 
\hline
  &   &  DFM2 &     &    & DFM2AR &    & \\
  \hline
$T$ & $m$ & GDPC & FHLR & SW & GDPC & FHLR & SW \\ 
  \hline
100 & 100 & 0.0552 & 0.0684 & 0.0727 & 0.0663 & 0.0832 & 0.0871 \\ 
   & 200 & 0.0465 & 0.0554 & 0.0562 & 0.0612 & 0.0699 & 0.0708 \\ 
   & 400 & 0.0435 & 0.0487 & 0.0478 & 0.0575 & 0.0635 & 0.0631 \\ 
  200 & 100 & 0.0319 & 0.0490 & 0.0515 & 0.0387 & 0.0568 & 0.0593 \\ 
   & 200 & 0.0253 & 0.0359 & 0.0357 & 0.0336 & 0.0433 & 0.0431 \\ 
   & 400 & 0.0237 & 0.0292 & 0.0278 & 0.0310 & 0.0364 & 0.0351 \\ 
   \hline
\end{tabular}
\caption{Means of the normalized MSEs of the estimation of the common part.} 
\label{tab:sim:DFM:MSE}
\end{table}

\FloatBarrier
To summarise the results:
\begin{itemize}
\item As expected, as $T$ and $m$ grow, the error in the estimation of the common part tends to diminish for all the estimators.
\item All estimators have a larger estimation MSE when dynamics are incorporated into the idiosyncratic part.
\item The MSE of GDPC is generally lower than those of the competitors.
\end{itemize}

\FloatBarrier

\section{Conclusions}
\label{sec:conclu}

In this short note, we proved a consistency result for the GDPC when the data follows a DFM. The results of a simulation study support our findings, moreover, the results show that the estimations of the common part provided GDPC compare favourably to those obtained by methods directly based on DFMs.

In the simulation study reported in \cite{PenaYohai2016} the authors find that the fact that GDPCs are not forced to be linear combinations of the data make them adapt well to non-stationarity. Extending the results of this paper for non-stationary DFMs is an interesting problem, left for future work.

\section{Appendix}
\label{sec:proofs}
In what follows $\Vert\cdot\Vert$ will stand for the Euclidean norm for vectors and
the spectral norm for matrices, $\Vert\cdot\Vert_{F}$ will stand for
the Frobenius norm for matrices and $\Tr(\cdot)$ will denote the trace.
Let $\hfb_{t}=(\widehat{f}_{t}, \dots, \widehat{f}_{t-\hk})^{\prime}$.
Let $\Zb$, $\Eb$, $\chib$, $\hZb$, $\hFb$ and $\Fb$ be the matrices with rows $\zb_{t}^{\prime}$, $\eb_{t}^{\prime}$, $\chib_{t}^{\prime}$, $\hzb_{t}(\hfb, \halb,\hBb)^{\prime}$, $\hfb_{t}^{\prime}$ and $\fb_{t}^{\prime}$ respectively.

\begin{lemma}
\label{lemma:errorrate}
Assume Conditions \ref{condition:DFM:consis} (b) and (c) hold. Then 
$$\frac{\Vert \Eb \Vert}{\sqrt {T m} } = O_{P}\left(\frac{1}{\min\left(T^{1/4}, m^{1/4} \right)}\right).
$$
\end{lemma}
\begin{proof}[Proof of Lemma \ref{lemma:errorrate}]
Take $\mathbf{v} \in \mathbb{R}^{m}$ with $\Vert \mathbf{v}\Vert = 1$. Then, using the Cauchy-Schwartz inequality we get
\begin{align*}
\frac{\mathbf{v}^{\prime} \Eb^{\prime} \Eb \mathbf{v}}{Tm} &= \frac{1}{Tm}\sum\limits_{i=1}^{m} \sum\limits_{j=1}^{m}\sum\limits_{t=1}^{T} v_{i}v_{j}e_{t,i}e_{t,j} 
\\ &\leq \left( \sum\limits_{i=1}^{m} \sum\limits_{j=1}^{m} v_{i}^{2}v_{j}^{2}\right)^{1/2}  \left(\frac{1}{m^{2}} \sum\limits_{i=1}^{m} \sum\limits_{j=1}^{m} \left(\frac{1}{T}\sum_{t=1}^{T}e_{t,i}e_{t,j}\right)^{2} \right)^{1/2}
\\ & = \left(\frac{1}{m^{2}} \sum\limits_{i=1}^{m} \sum\limits_{j=1}^{m} \left(\frac{1}{T}\sum_{t=1}^{T}e_{t,i}e_{t,j}\right)^{2} \right)^{1/2}
\end{align*}
and hence
$$
\left( \sup_{\Vert \mathbf{v} \Vert=1} \frac{\mathbf{v}^{\prime} \Eb^{\prime} \Eb \mathbf{v}}{Tm} \right)^{2} = \left(\frac{\Vert \Eb \Vert^{2}}{Tm}\right)^{2} \leq \frac{1}{m^{2}} \sum\limits_{i=1}^{m} \sum\limits_{j=1}^{m} \left(\frac{1}{T}\sum_{t=1}^{T}e_{t,i}e_{t,j}\right)^{2}
$$
Note that
\begin{align*}
\left(\frac{1}{T}\sum_{t=1}^{T}e_{t,i}e_{t,j}\right)^{2} = \frac{1}{T^{2}}\sum_{t=1}^{T}\sum_{s=1}^{T}e_{t,i}e_{t,j}e_{s,i}e_{s,j}
\end{align*}
and hence
\begin{align*}
\mathbb{E}\frac{1}{m^{2}} \sum\limits_{i=1}^{m} \sum\limits_{j=1}^{m} \left(\frac{1}{T}\sum_{t=1}^{T}e_{t,i}e_{t,j}\right)^{2} &= \frac{1}{(Tm)^{2}} \sum\limits_{i=1}^{m} \sum\limits_{j=1}^{m}\sum_{t=1}^{T}\sum_{s=1}^{T} \mathbb{E}e_{t,i}e_{t,j}e_{s,i}e_{s,j}.
\\ &= \frac{1}{(Tm)^{2}} \sum\limits_{i=1}^{m} \sum\limits_{j=1}^{m}\sum_{t=1}^{T}\sum_{s=1}^{T} \left(\mathbb{E}e_{t,i}e_{s,i}\right)\left(\mathbb{E}e_{t,j}e_{s,j}\right) 
\\ & + \frac{1}{(Tm)^{2}} \sum\limits_{i=1}^{m} \sum\limits_{j=1}^{m}\sum_{t=1}^{T}\sum_{s=1}^{T} cov(e_{t,i} e_{s,i}, e_{t,j} e_{s,j}).
\end{align*}
Now
\begin{align*}
 \frac{1}{(Tm)^{2}} \sum\limits_{i=1}^{m} \sum\limits_{j=1}^{m}\sum_{t=1}^{T}\sum_{s=1}^{T} \left(\mathbb{E}e_{t,i}e_{s,i}\right)\left(\mathbb{E}e_{t,j}e_{s,j}\right) &=  \frac{1}{(Tm)^{2}}  \sum_{t=1}^{T}\sum_{s=1}^{T} \left(\mathbb{E}\sum\limits_{i=1}^{m} e_{t,i}e_{s,i}\right)\left(\mathbb{E}\sum\limits_{j=1}^{m} e_{t,j}e_{s,j}\right)
 \\ &= \frac{1}{T^{2}} \sum_{t=1}^{T}\sum_{s=1}^{T} \left(\frac{1}{m}\mathbb{E}\sum\limits_{i=1}^{m} e_{t,i}e_{s,i}\right)^2
 \\ &= \frac{1}{T^{2}} \sum_{t=1}^{T}\sum_{u=1-t}^{T-t} \left(\frac{1}{m}\mathbb{E}\eb^{\prime}_{t} \eb_{t+u}\right)^2
  \\ &\leq \frac{1}{T^{2}} \sum_{t=1}^{T}\sum_{u\in\mathbb{Z}} \left(\frac{1}{m}\mathbb{E}\eb^{\prime}_{t} \eb_{t+u}\right)^2
 \\ &= \frac{1}{T} \sum_{u\in\mathbb{Z}} \left(\frac{1}{m}\mathbb{E}\eb^{\prime}_{1} \eb_{1+u}\right)^2
 \\ &= O\left(\frac{1}{T}\right),
\end{align*}
where we have used the stationarity of $\eb_{t}$ and Condition \ref{condition:DFM:consis} (b). On the other hand, by Condition \ref{condition:DFM:consis} (c)
\begin{align*}
\left\vert \frac{1}{(Tm)^{2}} \sum\limits_{i=1}^{m} \sum\limits_{j=1}^{m}\sum_{t=1}^{T}\sum_{s=1}^{T} cov(e_{t,i} e_{s,i}, e_{t,j} e_{s,j}) \right\vert \leq
 \frac{1}{m} \sup_{t,s \in\mathbb{Z}}\frac{1}{m} \sum\limits_{i=1}^{m} \sum\limits_{j=1}^{m} \left\vert cov(e_{t,i} e_{s,i}, e_{t,j} e_{s,j})\right\vert = O\left(\frac{1}{m}\right).
\end{align*}
We have thus shown that
$$
 \left( \frac{\Vert \Eb\Vert}{\sqrt{Tm}} \right)^{4}=O_{P}\left( \frac{1}{\min\left( T, m \right)}\right),
$$
from which the desired result follows immediately.
\end{proof}

\begin{proof}[Proof of Theorem \ref{theo:consis}]
Let $\fb=(f_{1-k},\dots,f_{T})^{\prime}$. The definition of $\hzb_{t}(\hfb, \halb,\hBb)$ implies
\begin{align*}
\frac{1}{Tm}\sum\limits_{t=1}^{T}\Vert \zb_{t} - \hzb_{t}(\hfb, \halb,\hBb))\Vert^{2}& = \frac{1}{Tm}\sum\limits_{t=1}^{T}\Vert  \chib_{t} + \eb_{t} - \hzb_{t}(\hfb, \halb,\hBb)\Vert^{2} 
\\ &\leq \frac{1}{Tm}\sum\limits_{t=1}^{T}\Vert  \chib_{t} + \eb_{t} - \hzb_{t}(\fb, \mathbf{0},\Bb)\Vert^{2} 
\\ &= \frac{1}{Tm}\sum\limits_{t=1}^{T}\Vert  \chib_{t} + \eb_{t} - \Bb^{\prime}\fb_{t} \Vert^{2} 
\\ & = \frac{1}{Tm}\sum\limits_{t=1}^{T}\sum\limits_{j=1}^{m} e_{t,j}^{2}.
\end{align*}
Hence, the reconstruction MSE is bounded by the mean square of the idiosyncratic part.
Note that
\begin{align*}
\frac{1}{Tm}\sum\limits_{t=1}^{T}\Vert \zb_{t} - \hzb_{t}(\hfb, \halb,\hBb))\Vert^{2} 
&= \frac{1}{Tm} \Vert \Zb - \hZb \Vert_{F}^{2} 
= \frac{1}{Tm} \Vert \chib + \Eb-\hZb\Vert_{F}^{2}
\\ &= \frac{\Vert \Eb \Vert_{F}^{2}}{Tm} + \frac{\Vert \chib- \hZb \Vert_{F}^{2}}{Tm} -  \frac{2\langle \Eb, \hZb - \chib \rangle}{Tm}
\\ &= \frac{1}{Tm}\sum\limits_{t=1}^{T}\sum\limits_{j=1}^{m} e_{t,j}^{2} 
\\ &+ \frac{1}{Tm}\sum\limits_{t=1}^{T}\Vert \chib_{t} - \hzb_{t}(\hfb, \halb,\hBb))\Vert^{2} 
\\ &- \frac{2\langle \Eb, \hZb - \chib \rangle}{Tm}.
\end{align*}
Thus, to prove the theorem it will be enough to show that
$2\langle \Eb, \hZb - \chib \rangle/(Tm)= O_{P}\left(1/\min\left(T^{1/4}, m^{1/4} \right) \right)$.

Note that $\Vert \Zb\Vert \leq \Vert \Fb \Vert \Vert \Bb \Vert + \Vert \Eb\Vert$. By Condition \ref{condition:DFM:consis} a), $\Vert \Bb\Vert = O(\sqrt{m})$, and since $\fb_{t}$ is second order stationary, we have that $\Vert \Fb \Vert = O_{P}(\sqrt{T})$. By Lemma \ref{lemma:errorrate}, $\Vert \Eb \Vert = o_{P}(\sqrt{Tm})$. Hence $\Vert \Zb \Vert = O_{P}(\sqrt{Tm})$ and $\Vert \chib \Vert=O_{P}(\sqrt{Tm}) $.  
Since $\hZb$ is obtained by projecting $\Zb$ on the space generated by the columns of $\hFb$, $\Vert \hZb\Vert \leq \Vert \Zb \Vert$.
Note also that both $\chib$ and $\hZb$ have rank bounded by $\widehat{r}=\hk +1$.
Hence, it will suffice to show that for each fixed $C>0$,
$$
\sup\limits_{\Vert \mathbf{L} \Vert \leq \sqrt{Tm} C, \: rank(\mathbf{L})\leq \widehat{r}} \left\vert \frac{\langle \Eb, \mathbf{L}\rangle}{Tm} \right\vert = O_{P}\left(\frac{1}{\min\left(T^{1/4}, m^{1/4} \right)} \right).
$$
Take $\mathbf{L}$ with $\Vert \mathbf{L} \Vert \leq \sqrt{Tm} C, \: rank(\mathbf{L})\leq \widehat{r}$. Using the Singular Value Decomposition, write
$
\mathbf{L} = \sum_{l=1}^{\widehat{r}} \sigma_{l} \mathbf{u}_{l} \mathbf{v}_{l}^{\prime},
$
with $\mathbf{u}_{l}, \mathbf{v}_{l}$ having unit norm. Then
\begin{align*}
\left\vert \frac{\langle \Eb, \mathbf{L}\rangle}{Tm} \right\vert \leq \frac{1}{Tm} \sum\limits_{l=1}^{\widehat{r}} \sigma_{l} \vert \langle \Eb, \mathbf{u}_{l} \mathbf{v}_{l}^{\prime} \rangle\vert &\leq \frac{1}{Tm} \max_{l\leq \widehat{r}}\sigma_{l} \sum\limits_{l=1}^{\widehat{r}}  \vert \langle \Eb, \mathbf{u}_{l} \mathbf{v}_{l}^{\prime} \rangle\vert 
\\&=  \frac{1}{Tm} \Vert \mathbf{L}\Vert \sum\limits_{l=1}^{\widehat{r}}  \vert \langle \Eb, \mathbf{u}_{l} \mathbf{v}_{l}^{\prime} \rangle\vert .
\end{align*}
Now take any $\mathbf{u}, \mathbf{v}$ having unit norm. Then, using the cyclic property of the trace
$$
\vert \langle \Eb, \mathbf{u} \mathbf{v}^{\prime} \rangle\vert = \left\vert \Tr( \Eb \mathbf{v}\mathbf{u}^{\prime} ) \right\vert= \left\vert \Tr( \mathbf{u}^{\prime}\Eb \mathbf{v} ) \right\vert \leq \Vert \Eb \Vert.
$$
Using Lemma \ref{lemma:errorrate}, we have
\begin{align*}
\left\vert \frac{\langle \Eb, \mathbf{L}\rangle}{Tm} \right\vert \leq \frac{\widehat{r} \Vert \Eb \Vert \Vert \mathbf{L} \Vert}{Tm}  \leq  \frac{\widehat{r} \Vert \Eb \Vert C}{\sqrt{Tm} } = O_{P}\left(\frac{1}{\min\left(T^{1/4}, m^{1/4} \right)} \right),
\end{align*}
uniformly in $\mathbf{L}$, from which the Theorem follows.
\end{proof}

\bibliographystyle{apalike}
\bibliography{fore}

\end{document}